\documentclass[11pt, a4paper]{article}

\usepackage[utf8]{inputenc}
\usepackage[T1]{fontenc}

\usepackage{graphicx}
\usepackage{amsmath}
\usepackage{amsthm}
\usepackage{amsfonts}
\usepackage{xcolor}
\usepackage{amssymb}
\usepackage{tikz}
\usepackage[hyphens]{url}
\usepackage[hidelinks]{hyperref}
\usepackage[hyphenbreaks]{breakurl}
\usepackage{fullpage}
\usepackage{esvect,mathtools}
\usepackage{subfiles}
\usepackage[normalem]{ulem}
\usepackage{enumitem}
\usepackage[capitalize]{cleveref}

\newtheorem{theorem}{Theorem}
\numberwithin{theorem}{section} 
\newtheorem{lemma}[theorem]{Lemma}
\newtheorem{proposition}[theorem]{Proposition}

\newtheorem{conjecture}[theorem]{Conjecture}
\newtheorem{claim}[theorem]{Claim}
\theoremstyle{remark}

\theoremstyle{definition}

\crefname{lemma}{lemma}{lemmas}
\Crefname{lemma}{Lemma}{Lemmas}

\def\cA{\mathcal{A}}
\def\cC{\mathcal{C}}

\def\cF{\mathcal{F}}

\def\cH{\mathcal{H}}

\def\cS{\mathcal{S}}

\def\bR{\mathbb{R}}

\def\Pb{\mathbb{P}}

\title{Spanning subhypergraphs with degree constraints}

\date{}

\author{Noga Alon \thanks{Department of Mathematics, Princeton University and Schools of Mathematics, Computer Science and AI, Tel Aviv University, Israel. Research supported in part by NSF grant DMS-2553988 and by AFOSR grant 26RT0135. {Email}: {\tt nalon@math.princeton.edu}}
\and Penny Haxell \thanks{Department of Combinatorics and Optimization, University of Waterloo, Waterloo ON, Canada. Research supported in part by NSERC. {Email:} {\tt pehaxell@uwaterloo.ca}}
\and Aleksa Milojevi\'c\thanks{Department of Mathematics, ETH, Z\"urich, Switzerland. Research supported in part by SNSF grant 200021-228014. {Email}: {\tt aleksa.milojevic@math.ethz.ch}} \and
Jacques Verstra\"ete\thanks{Department of Mathematics, University of California, San Diego, CA, USA. Research supported in part by NSF grant DMS-2347832. {Email:} {\tt jacques@ucsd.edu}}}

\begin{document}

\maketitle

\begin{abstract}
An old result of Tutte states that any $d$-regular graph contains a spanning subgraph in which every vertex has degree $k$ or $k+1$, for every $1\leq k\leq d$. We generalize this statement to hypergraphs, showing, for example, that every $3$-uniform $d$-regular hypergraph contains a subgraph in which all degrees are $k, k+1$ or $k+2$, for every $1\leq k\leq d$. This statement is best possible in the sense that the corresponding statement with only two allowed consecutive values is not true. We provide generalizations of this statement to higher uniformities and discuss several open problems. 
\end{abstract}

\section{Introduction}

Many classical problems in graph theory can be formulated as questions about finding a subgraph with prescribed degree conditions. For example, two basic theorems of graph theory, Hall's marriage theorem \cite{Hall35} and Tutte's matching theorem \cite{Tutte47}, give necessary and sufficient conditions for graphs to have a perfect matching, which is a spanning subgraph in which all degrees are one. If one asks the spanning subgraph to be 2-regular instead of 1-regular, then Petersen's theorem \cite{Petersen91} guarantees that such subgraphs can be found in all even-regular graphs.

The general framework for studying questions about subgraphs with prescribed degrees was introduced by Lov\'asz \cite{Lovasz70Valencies,Lovasz72}. Suppose that every vertex $v\in V(G)$ is given a list of allowed degrees $H(v)$. Under which conditions on the lists and on $G$ does there exist a subgraph $G'\subseteq G$ for which $\deg_{G'} v\in H(v)$ for all $v\in V(G)$? This general problem was studied in a variety of contexts, including \cite{AddarioBerryDalalReed08, BergeLasVergnas, Cornuejols88, ShiraziVerstraete08}.

Regular graphs have particularly nice properties in this context. For example, Hall's theorem shows that regular bipartite graphs always have perfect matchings, and Petersen's theorem shows that even-regular graphs always have $2$-factors. Moreover, as Tutte \cite{Tutte53} showed, any regular graph has a spanning subgraph in which all degrees are either $1$ or $2$. There is no hope of proving such statements without some regularity assumption, since the graph may be a star.

In this paper, we focus on a version of this problem for regular hypergraphs. The natural way to generalize Tutte's theorem is to ask for which values of $t$ and $r$ it is true that any regular $r$-uniform hypergraph $\cH$ contains a spanning subhypergraph all of whose degrees are in the set $\{1, 2, \dots, t\}$.

The first interesting case of this question concerns regular $3$-uniform hypergraphs, where the answer turns out to be $t=3$. This is tight, as the example of the Fano plane shows.

\begin{theorem}\label{thm:3-uniform}
Let $\cH$ be a nonempty regular $3$-uniform hypergraph. Then, there is a spanning subhypergraph $\cC\subseteq \cH$ such that $\deg_{\cC} v\in \{1, 2, 3\}$ for all $v\in V(\cH)$.
\end{theorem}

A natural way to extend the above theorem is to ask about higher uniformities. In this direction, for $r\geq 2$, we define $f(r)$ to be the smallest value of $t$ for which any nonempty regular $r$-uniform hypergraph contains a spanning subhypergraph with all degrees between $1$ and $t$. Tutte's theorem shows that $f(2)=2$ and Theorem~\ref{thm:3-uniform} shows that $f(3)=3$. The following result determines the asymptotics of $f(r)$ for growing $r$.

\begin{theorem}\label{thm:r-uniform}
Let $r\geq 2$ be an integer, and let $\cH$ be a nonempty regular $r$-uniform hypergraph. Then, there exists a spanning subhypergraph $\cC\subseteq \cH$ such that $1\leq \deg_{\cC} v\leq 100\log r$. In other words $f(r)\leq 100\log r$.
\end{theorem}

Planken and Ueckerdt \cite{PlankenUeckerdt23} have constructed examples which give $f(r)\geq \lfloor \log_2(r+1)\rfloor $, showing that Theorem~\ref{thm:r-uniform} is tight up to a constant factor.

Further, note that we always have $f(r+1)\geq f(r)$. To see this, suppose that $\cH$ is an $r$-uniform hypergraph without a spanning subhypergraph all of whose degrees are in $\{1, \dots, t\}$, for some $t$. Then, one can construct an $(r+1)$-uniform hypergraph $\cH'$ with the same property. This can be done as follows: denote $S=V(\cH)$ and add a new vertex outside $S$ to each edge of $\cH$. Then, add further vertices and edges disjoint from $S$ until one obtains a regular $(r+1)$-uniform hypergraph $\cH'$. If $\cH'$ has a spanning subhypergraph all of whose degrees are between $1$ and $t$, then intersecting the edges of this hypergraph with $S$ would give a subhypergraph of $\cH$ with the same property, which we assumed did not exist.

\medskip

One can also attempt to generalize Theorem~\ref{thm:3-uniform} in a different direction. To explain how, let us go back to graphs for a moment, where we have already mentioned that any regular graph contains a $\{1, 2\}$-factor. An even stronger statement is true: Tutte showed that for any $d\geq k\geq 1$, any $d$-regular graph contains a $\{k, k+1\}$-factor (see also \cite{Thomassen81} for a very nice proof). Our next result shows that a similar statement is true for $3$-uniform hypergraphs.   

\begin{theorem}\label{thm:3-uniform-general}
Let $d\geq k\geq 1$ be integers, and let $\cH$ be a $d$-regular $3$-uniform hypergraph. Then, there is a spanning subhypergraph $\cC\subseteq \cH$ such that $\deg_{\cC} v\in \{k, k+1, k+2\}$ for all $v\in V(\cH)$. 
\end{theorem}

If one tries to extend this result to higher uniformities, one encounters a significantly different situation from Theorem~\ref{thm:r-uniform}. To be precise, for $r\geq 2$, we define $g(r)$ to be the smallest value of $t$ such that for any $d\geq k\geq 1$ and any $d$-regular $r$-uniform hypergraph $\cH$, there exists a spanning subhypergraph $\cC\subseteq \cH$ such that for every vertex $v\in V(\cH)$ we have $\deg_{\cC} v\in \{k, k+1, \dots, k+t-1\}$. Again, Tutte's theorem \cite{Tutte78} and Theorem~\ref{thm:3-uniform-general} show that $g(2)=2$ and $g(3)=3$. Observe that, using the same argument as in the case of $f(r)$, one can show that $g$ is a nondecreasing function, i.e. that $g(r+1)\geq g(r)$ for all $r$. The following result does not describe the precise asymptotics of $g(r)$,
but shows that for large $r$ it is much larger than $f(r)$.

\begin{theorem}\label{thm:r-uniform-general}
For any $r\geq 3$, we have $\lfloor \sqrt{r/2}\rfloor\leq g(r)\leq 2r-3$.
\end{theorem}

Note that the upper bound $g(r)\leq 2r-3$ immediately implies Theorems~\ref{thm:3-uniform} and~\ref{thm:3-uniform-general} by setting $r=3$. Thus, we will only focus on proving Theorems~\ref{thm:r-uniform} and~\ref{thm:r-uniform-general}. Our proofs of the upper bounds in the above theorems are algorithmic in nature, meaning that they provide an efficient procedure to find a subhypergraph with required degree conditions.

The paper \cite{GFL26} describing an AI-generated 
proof of the Koml\'os Conjecture was posted in the arXiv essentially simultaneously with our manuscript.
Corollary 1.2 in that paper implies that the
hereditary discrepancy of any set system in which every point belongs
to at most $r$ sets is smaller than $3\sqrt{2\pi r}$.
Combining it with the known relation between the linear discrepancy and
the hereditary discrepancy of set systems (see, e.g., \cite{AlonSpencer}, 
Theorem
13.3.2) this implies that the function $g(r)$
satisfies 

$$g(r) \leq 6\sqrt{2 \pi r}.$$ 

This improves the bound in
Theorem \ref{thm:r-uniform-general}
for all $r \geq 60$, and shows that for large
$r$, $g(r)=\Theta(\sqrt r)$. The proof in \cite{GFL26} does not 
provide an efficient algorithm.

One can also extend Theorem~\ref{thm:r-uniform-general} to arbitrary sets of allowed degrees. Namely, if $\cH$ is a $d$-regular $r$-uniform hypergraph and every vertex $v$ is assigned a list $H(v)$ of more than $\lceil \frac{r-1}{r}d\rceil$ allowed degrees between $0$ and $d$, then a simple modification of the Shirazi-Verstra\"ete argument from \cite{ShiraziVerstraete08} shows that there is a spanning subhypergraph $\cH'\subseteq \cH$ for which $\deg_{\cH'}v\in H(v)$ for all vertices $v\in V(\cH)$. This proof is based on the Combinatorial Nullstellensatz \cite{Alon99} and so it does not provide an efficient algorithm for finding this subhypergraph.

During the work on this paper we also found a proof of the weaker result that for $r=3$ one can find a subhypergraph with all degrees in $\{1, 2, 3, 4\}$, using a completely different argument that is ultimately based on topological methods. For completeness, we present this different perspective as well. 

\subsection{Related work}

A spanning subhypergraph $\cC\subseteq \cH$ in which all degrees are between $1$ and $t$ is closely related to the notion of $t$-shallow hitting sets introduced by Keszegh and P\'alv\"olgyi \cite{KeszeghPalvolgyi19}. They were initially interested in the polycolorability of geometric hypergraphs, where the question is to determine for a given $r$-uniform hypergraph $\cH$ (usually coming from a geometric setting) the maximum $k$ for which vertices of $\cH$ can be colored using $k$ colors such that every edge of $\cH$ receives all colors. A natural way to construct such a coloring is to find a set $S\subseteq V(\cH)$ which intersects every hyperedge in at least $1$ and at most $t$ elements. Setting these vertices aside as one color class, we are then left with a hypergraph in which all edges have size at least $r-t$, and if one can iterate this procedure, it is possible to obtain a coloring using $\lfloor r/t\rfloor$ colors in which every hyperedge contains all colors.

A set $S$ intersecting every hyperedge in at least $1$ and at most $t$ vertices is then called a \textit{$t$-shallow hitting set} (see \cite{PlankenUeckerdt23}). A natural dual notion is a set of edges $F\subseteq E(\cH)$ which covers every vertex at least once and at most $t$ times. We call such an $F$ a \textit{$t$-shallow hitting edge set}.

Planken and Ueckerdt \cite{PlankenUeckerdt23} have studied when $t$-shallow hitting edge sets exist in a variety of settings, including the setting of regular $r$-uniform hypergraphs. In particular, they have shown that $f(r)\geq \lfloor \log_2(r+1)\rfloor $ and $f(r)\leq O(r)$. Theorem~\ref{thm:r-uniform} improves the upper bound and shows that $f(r)\leq O(\log r)$, thus obtaining a bound which is tight up to a constant factor.

Shallow hitting edge sets seem to be useful for finding other spanning structures in dense hypergraphs, see e.g. \cite{LangSanhuezaMatamala23}.

\subsection{Further questions}

Let us now discuss some questions in similar spirit that we could not answer precisely. 

Firstly, it seems possible that the conclusion of Theorem~\ref{thm:3-uniform} could be strengthened if one restricts to tripartite hypergraphs (observe that the Fano plane, which shows the tightness of Theorem~\ref{thm:3-uniform}, is not tripartite). 

\begin{conjecture}\label{conj:tripartite}
Let $\cH$ be a nonempty regular tripartite $3$-uniform hypergraph. Then, $\cH$ contains a spanning subhypergraph whose vertices each have degree either $1$ or $2$.
\end{conjecture}

Although we are unable to show this statement in general, as partial evidence towards it we provide a simple argument based on Seymour's work \cite{Seymour74} that proves the conjecture for $3$-regular tripartite hypergraphs. 

Secondly, recall that Theorem~\ref{thm:3-uniform} shows that any regular $3$-uniform hypergraph has a spanning subgraph all of whose degrees are in $\{1, 2, 3\}$. What can be said about regular $4$-uniform hypergraphs? What is the value of $f(4)$? We know that $f(4)\leq 2\cdot 4-3=5$ from Theorem~\ref{thm:r-uniform-general} and $f(4)\geq f(3)=3$ by the monotonicity properties of $f$.

While we can say little about this question in general, we note that an old result of Thomassen shows that any $4$-regular $4$-uniform hypergraph has property B (see Theorem 5.1 in \cite{Thomassen92}). The dual of this statement says that the edge set of any $4$-regular $4$-uniform hypergraph $\cH$ can be partitioned into $E(\cH)=E_1\cup E_2$, where each vertex is covered between $1$ and $3$ times. In other words, there exists a spanning subhypergraph $\cC\subseteq \cH$ for which $\deg_{\cC} v\in \{1, 2, 3\}$. This suggests that we perhaps have $f(4)=3$, although this may be too specific a case to be an indicator of general conclusions.

\medskip
\noindent
\textbf{Paper organization.} We prove the upper bound from Theorem~\ref{thm:r-uniform-general} in Section~\ref{sec:upper_bounds}, and we note that this immediately implies Theorem~\ref{thm:3-uniform} and~\ref{thm:3-uniform-general}. We also prove Theorem~\ref{thm:r-uniform} in Section~\ref{sec:upper_bounds}. Then we present an argument proving Conjecture~\ref{conj:tripartite} for 3-regular hypergraphs in Section~\ref{sec:tripartite}. Our alternative argument for the weaker version of Theorem~\ref{thm:3-uniform} is given in the appendix.

\section{Proofs of Theorems~\ref{thm:r-uniform} and~\ref{thm:r-uniform-general}}\label{sec:upper_bounds}

In this section, we show that every $d$-regular $r$-uniform hypergraph contains a spanning subhypergraph with all degrees in the set $\{k, \dots, k+2r-4\}$, for any $k\leq d$. Let us now briefly outline the proof ideas needed for this result. Our proof of the upper bound $g(r)\leq 2r-3$ is inspired by the proof of the Beck-Fiala theorem \cite{BeckFiala81}, which states that if a set system $\cS$ on the ground set $X$ has the property that no element is contained in more than $n$ sets, then one can color the elements of $X$ red and blue so that the difference between the number of red and blue elements in any set $A\in \cS$ is at most $2n-1$. The proof is based on the floating-variables method, in which every element $v\in X$ is assigned a real number $x_v\in [-1, 1]$, with all numbers starting at $0$. Then, the variables are gradually changed with the goal of each of them reaching either $1$ (at which point we color the element red and do not change $x_v$ further) or $-1$ (at which point we color the element blue and do not change $x_v$ further). Throughout this process, one maintains $\sum_{v\in A} x_v=0$ for all $A\in \cS$, in order to ensure that the number of red and blue elements is balanced in each set. For the full proof of the Beck-Fiala theorem, see e.g. Section 13 of \cite{AlonSpencer}.

Our proof of Theorem~\ref{thm:r-uniform-general} proceeds along very similar lines, and we explain it for $3$-uniform hypergraphs for the sake of concreteness. The goal is to show that any $d$-regular hypergraph contains a spanning subhypergraph with degrees in the set $\{k, k+1, k+2\}$, if $d\geq k$. We begin by assigning a variable $x_e\in [0, 1]$ to every edge $e\in E(\cH)$, starting at $x_e=\frac{k+1}{d}$ (observe that if $d=k$ or $d=k+1$, then the statement is trivial, so we may assume $d\geq k+2$, so $0<x_e<1$). We proceed to move these values until one of them reaches either $0$ or $1$, at which point the edge is either excluded or included in our subhypergraph $\cC$ and the variable $x_e$ is not moved further. As before, throughout this process of moving the variables, we ensure that $\sum_{e\ni v}x_e=k+1$ for all vertices $v\in V(\cH)$.

As long as the number of free variables coming from edges with $x_e\in (0, 1)$ is larger than the number of constraints coming from vertices, we can find a direction to move the variables until one of them reaches either $0$ or $1$. Crucially, in order to reduce the number of constraints as the process progresses, we observe that any vertex with two nonfixed edges remaining will, no matter what values are assigned to these nonfixed edges, have between $k$ and $k+2$ incident edges at the end. Thus, as soon as all but two edges incident to a vertex are fixed, we can forget about the constraint corresponding to this vertex. As a consequence of a simple double counting, the only place where this process can get stuck is if we are left with a $3$-regular hypergraph, from which we can then choose the remaining edges directly by a simple procedure. Finally, Theorem~\ref{thm:r-uniform} can be derived from the above statement using a simple application of the Lov\'asz Local Lemma.

This method is quite robust, and it allows for the assumptions to be 
significantly relaxed. For example, it is sufficient to assume that 
the hypergraphs in question have rank $r$ instead of uniformity $r$ 
(in other words, we can allow for edges of cardinality smaller than $r$). 
Additionally, the only place where the regularity assumption is used 
is to construct a fractional spanning subgraph with all
fractional degrees $k+r-2$, which is an assignment of weights 
$w_e\in [0, 1]$ to all edges $e\in E(\cH)$ satisfying 
$\sum_{e\ni v} w_e=k+r-2$ for all $v\in V(\cH)$. 
Hence, Theorems~\ref{thm:3-uniform}--\ref{thm:r-uniform-general} apply
to all hypergraphs with such a fractional spanning subgraph, or
equivalently, all $r$-uniform hypergraphs with a
perfect fractional matching in which all weights are
at most $1/(k+r-2)$.

\begin{proof}[Proof of the upper bound in Theorem~\ref{thm:r-uniform-general}.]
Given $\cH$, finding a subhypergraph $\cC\subseteq \cH$ that witnesses $g(r)\leq 2r-3$ for $\cH$ corresponds to finding a vector of weights $x \in \{0, 1\}^{E(\cH)}$ such that for every vertex $v\in V(\cH)$ we have $\sum_{e\ni v} x_e\in \{k, \dots, k+2r-4\}$ (where $x_e=1$ means that the edge $e$ is included in the cover $\cC$ and $x_e=0$ means that it is not). Note that we may assume that $d\geq k+2r-3$, as otherwise taking $\cC=\cH$ already satisfies the needed conditions.

We find these values using the following algorithm. We start by setting $x_e=\frac{k+r-2}{d}$ for each $e\in E(\cH)$ and we will vary these variables until we arrive at a $\{0, 1\}$-valued vector satisfying $\sum_{e\ni v} x_e\in \{k, \dots, k+2r-4\}$ for each $v\in V(\cH)$.

The algorithm maintains a partition of the edge set into \textit{active} and \textit{fixed} edges, $E(\cH)=E_a\cup E_f$, and a partition of the vertex set into \textit{safe} and \textit{unsafe} vertices, $V(\cH)=V_s\cup V_u$.

The active edges will be the ones having $0<x_e<1$, and these values will be changed throughout the rest of the algorithm. On the other hand, when a variable $x_e$ reaches value $0$ or $1$, it becomes fixed and remains fixed for the rest of the process. If a vertex $v$ has at most $r-1$ active incident edges, then we declare $v$ to be \textit{safe} until the end of the algorithm.

Throughout the algorithm, we maintain the following set of constraints.
\[\sum_{e\ni v} x_e=k+r-2\text{ for every }v\in V_u\text{ and }0\leq x_e\leq 1\text{ for every }e\in E(\cH).\]

In a general step of the algorithm, as long as the number of active edges exceeds the number of vertex constraints, i.e. as long as $|E_a|>|V_u|$, we do the following. We find a solution to the system of equations \[\sum_{\substack{e\ni v\\e\in E_a}} y_e=0\text{ for every }v\in V_u.\]
If $|E_a|>|V_u|$, this system is underdetermined and so there exists a nontrivial solution $(y_e)_{e\in E_a}$ to this system. We then take this solution, define $y_e=0$ for all fixed edges and update the variables $x_e\mapsto x_e+\lambda y_e$ for some real number $\lambda$. Observe that the new values of $x_e$ still satisfy $\sum_{e\ni v} x_e=k+r-2$ for every $v\in V_u$. Moreover, to choose the value of $\lambda$, we start at $\lambda=0$ and increase it until one (or more) of the variables $x_e$ becomes $0$ or $1$. At this point, we fix $x_e$, we check if there is any vertex $v$ which becomes safe, and proceed to the next step of the algorithm. The algorithm terminates when $|E_a|=|V_u|$.

Let us make a couple of observations about this algorithm. First of all, we observe that if a vertex $v$ is declared to be safe, then no matter whether the remaining active edges incident to $v$ are included in the cover or not, $v$ will have degree between $k$ and $k+2r-4$ in the cover. To see this, we argue based on how many edges incident to $v$ have been fixed with $x_e=1$, at the moment when $v$ is declared to be safe. If all edges incident to $v$ are fixed when $v$ is declared safe, then $v$ will surely have degree $k+r-2$ in the final subhypergraph. So, suppose that this is not the case.

Since at least one and at most $r-1$ active edges are incident to $v$, we have $0<\sum_{\substack{e\ni v\\e\in E_a}} x_e<r-1$. Moreover, this quantity is an integer, and thus $\sum_{\substack{e\ni v\\e\in E_f}}x_e=k+r-2-\sum_{\substack{e\ni v\\e\in E_a}} x_e\in \{k, \dots, k+r-3\}$. In other words, between $k$ and $k+r-3$ fixed edges containing $v$ have $x_e=1$. Thus, no matter what happens with the remaining $r-1$ edges, $v$ will have degree between $k$ and $k+2r-4$ in the resulting subhypergraph.

Let $\cH'$ be the hypergraph of rank at most $r$ on the vertex set $V_u$, with edges given by $e\cap V_u$ for $e\in E_a$. Let us show that when the algorithm terminates with $|E_a|=|V_u|$, the hypergraph $\cH'$ is $r$-regular and $r$-uniform. Note that every vertex of $V_u$ is incident to at least $r$ edges of $E_a$, so the average degree is at least $r$. On the other hand, the average degree of $\cH'$ can be calculated as follows:
\[\bar{d}(\cH')=\frac{1}{|V_u|}\sum_{e\in E_a}|e\cap V_u|\leq \frac{r|E_a|}{|V_u|}\leq r.\]
Hence, we have the equality all the way through, meaning that each edge of $E_a$ is completely contained in $V_u$ and that $\cH'$ is an $r$-regular, $r$-uniform hypergraph.

So, when the algorithm terminates, we have obtained a partial subhypergraph $\cC_0$ consisting of those fixed edges having $x_e=1$. Let $f_v=\deg_{\cC_0} v$ denote the degree already assigned to the vertex $v$. It can be computed from the equation $f_v+\sum_{\substack{e\ni v\\e\in E_a}}x_e=k+r-2$. 

To complete the proof, we need to decide which of the remaining active edges we add to the set $\cC_0$ to form the final hypergraph $\cC$ having the property that $\deg_{\cC}v\in \{k, \dots, k+2r-4\}$. We do that as follows: if an edge has $x_e\geq 1/2$ when the above algorithm terminates, we include it in the final hypergraph $\cC$, and otherwise we leave it out.

The number of additional edges incident to $v$ differs from $\sum_{\substack{e\ni v\\e\in E_a}}x_e$ by at most $r/2$, since the rounding of every edge contributes at most $1/2$ to the difference. Hence, the degree of $v$ in the final hypergraph $\cC$ satisfies
\[|\deg_\cC v-(k+r-2)|\leq \Big|\deg_\cC v-f_v-\sum_{\substack{e\ni v\\e\in E_a}}x_e\Big|\leq r/2< r-1.\]
Since $\deg_\cC v$ is an integer, we conclude $\deg_{\cC} v \in \{k, \dots, k+2r-4\}$, thus completing the proof.
\end{proof}

We now discuss the proof of the lower bound in Theorem~\ref{thm:r-uniform-general}. It is a simple consequence of the following well-known lower bound construction in discrepancy theory (see e.g. Corollary 2.10 in \cite{BeckSos95}). 

\begin{theorem}\label{thm:discrepancy-lower-bound}
Let $\cH$ be a projective plane of order $q$, i.e. a $(q+1)$-uniform $(q+1)$-regular hypergraph on the ground set $[q^2+q+1]$ for which each pair of elements is contained together in exactly one set. Then, no matter how the vertices of $[q^2+q+1]$ are colored red and blue, there is an edge of $\cH$ which contains at least $\sqrt{q}$ more vertices of one color than the other. 
\end{theorem}

Recall that the \textit{dual} of a hypergraph $\cH$ is the hypergraph with vertex set $E(\cH)$ and edge set $\{\{e\ni v\}: v\in V(\cH)\}$. 

\begin{proof}[Proof of the lower bound in Theorem~\ref{thm:r-uniform-general}.]
Let $q$ be a prime between $r/2$ and $r-1$, and let $\cH$ be a projective plane of order $q$, which is a $(q+1)$-uniform $(q+1)$-regular hypergraph. By Theorem~\ref{thm:discrepancy-lower-bound}, no matter how one colors the vertices of $\cH$ red and blue, there always exists an edge in which there are at least $\sqrt{q}$ more red points than blue points (or vice versa). 

If $\cH'$ is the dual of the hypergraph $\cH$, note that its edges cannot be partitioned into two sets $E(\cH')=E_1\cup E_2$ such that each $\deg_{E_1} v\in \{(q+1)/2-\lfloor \sqrt{q}/2\rfloor, \dots, (q+1)/2+\lfloor \sqrt{q}/2\rfloor\}$ (otherwise, this would provide a balanced coloring of the vertices of the hypergraph $\cH$). This shows that $g(r)\geq g(q+1)\geq \lfloor \sqrt{q}\rfloor \geq \lfloor \sqrt{r/2}\rfloor$. 
\end{proof}

We close this section by giving the proof of Theorem~\ref{thm:r-uniform}, which states that regular $r$-uniform hypergraphs contain spanning subhypergraphs with all degrees between $1$ and $100\log r$. By passing to the dual, we see that Theorem~\ref{thm:r-uniform} is equivalent to the following statement. 
\begin{proposition}\label{prop:dual_problem}
If $\cF$ is an $r$-regular $d$-uniform hypergraph, then there exists a set of vertices $S\subseteq V(\cF)$ such that for any $e\in E(\cF)$ we have $1\leq |e\cap S|\leq 100\log r$.
\end{proposition}

We split the proof of this proposition into two parts: first, we discuss what happens when $d\leq O(r)$, in which an application of Lov\'asz's local lemma \cite{ErdosLovasz75} suffices to prove the statement. When $d$ is much larger than $r$, a simple application of Theorem~\ref{thm:r-uniform-general} allows us to reduce to the previous case. In the case $d\leq O(r)$, we prove the following statement.

\begin{lemma}\label{lemma:low_uniformity}
Let $\cF$ be an $r$-regular hypergraph, with all edges having size between $k$ and $3k$, for some $1\leq k\leq r$. Then there exists a set $S\subseteq V(\cF)$ such that $1\leq |e\cap S|\leq 100\log r$ for every $e\in E(\cF)$.
\end{lemma}
\begin{proof}
Let $S$ be a random set of vertices, obtained by including every $v\in V(\cF)$ independently with probability $\frac{10\log r}{k}$ (we may assume that $k\geq 10\log r$, as otherwise one may take $S=V(\cF)$). We aim to show that with positive probability one has $1\leq |e\cap S|\leq 100\log r$ for every $e\in E(\cF)$.

Our main goal will be to show that for every fixed hyperedge $e\in E(\cF)$ we have \[\Pb\big[1\leq |e\cap S|\leq 100 \log r\big]\geq 1- \frac{1}{10kr}.\] Suppose we have shown this. Then, a simple application of the local lemma suffices to complete the proof as follows. For every $e\in E(\cF)$, define the bad event $B_e$ as the complement of the event $1\leq |e\cap S|\leq 100\log r$. Since the event $B_e$ only depends on whether the vertices of $e$ are included in $S$ or not, it is independent of all $B_{e'}$ with $e\cap e'=\varnothing$. So, the maximum degree of the dependency graph is at most $3k\cdot r$, since $|e|\leq 3k$ and every vertex of $e$ is contained in at most $r$ other edges. Since ${\rm{e}}(3kr+1)\cdot \frac{1}{10kr}<1$, the local lemma implies that with positive probability no bad event $B_e$ occurs, showing that we have $1\leq |e\cap S|\leq 100\log r$ for every $e\in E(\cF)$.

So, it remains to bound $\Pb[1\leq |e\cap S|\leq 100\log r]$. To do this, note first that 
\[\Pb[e\cap S=\varnothing]\leq \Big(1-\frac{10\log r}{k}\Big)^k\leq {\rm{e}}^{-10\log r}=r^{-10}.\]
On the other hand, by the union bound, for any $t$ we have
\[\Pb\big[|e\cap S|\geq t\big]\leq \binom{3k}{t}\Big(\frac{10\log r}{k}\Big)^t\leq \Big(\frac{3{\rm{e}}k}{t}\Big)^t \Big(\frac{10\log r}{k}\Big)^t\leq  \Big(\frac{3{\rm{e}} \cdot 10\log r}{t}\Big)^t.\]
Putting $t=100\log r$ gives 
\[\Pb\big[|e\cap S|\geq 100\log r\big]\leq \Big(\frac{3{\rm{e}} \cdot 10\log r}{100\log r}\Big)^{100\log r}\leq 0.9^{100\log r}\leq r^{-10},\]
where we have used that $30{\rm{e}}/100<0.9$ and $0.9^{100}\leq e^{-10}$. By a simple union bound, we obtain that $\Pb[1\leq |e\cap S|\leq 100\log r]\geq 1-2r^{-10}\geq 1-\frac{1}{10kr}$, which completes the proof as we already explained.
\end{proof}

Let us now give the proof of Proposition~\ref{prop:dual_problem} (and thus of Theorem~\ref{thm:r-uniform}, which is its dual).

\begin{proof}[Proof of Theorem~\ref{thm:r-uniform} and Proposition~\ref{prop:dual_problem}.]
Let $\cH$ be the given $d$-regular $r$-uniform hypergraph and let $\cF$ be its dual. If we have $d\geq 3r$, then by Theorem~\ref{thm:r-uniform-general} applied with $k=r$ we find a spanning subhypergraph $\cC\subseteq \cH$ in which all degrees are in $\{r, \dots, 3r-4\}$. Taking its dual hypergraph $\cF'$ corresponds to restricting the set of vertices of $\cF$ such that all hyperedges have size between $r$ and $3r-4$. Hence, by Lemma~\ref{lemma:low_uniformity} applied with $k=r$ we find a subset $S\subseteq V(\cF')$ with the property that $1\leq |S\cap e|\leq 100\log r$ for all $e\in E(\cF')$. The set of hyperedges in $\cH$ dual to $S$ then gives a subhypergraph in which every vertex $v\in V(\cH)$ has degree between $1$ and $100\log r$.

In the case $d\leq 3r$ we argue similarly, by directly applying Lemma~\ref{lemma:low_uniformity} with $k=\lceil d/3\rceil$ to obtain a set $S$ of vertices for which $1\leq |S\cap e|\leq 100\log r$. As we noted above, considering its dual gives the subhypergraph we need in $\cH$.
\end{proof}

\section{Tripartite $3$-regular hypergraphs}\label{sec:tripartite}

In this section, we present an argument which proves Conjecture~\ref{conj:tripartite} for $3$-regular hypergraphs.

\begin{proposition}\label{prop:2-shallow}
Let $\cH$ be a $3$-regular tripartite $3$-uniform hypergraph. Then, $\cH$ has a spanning subhypergraph whose degrees are all either $1$ or $2$.
\end{proposition}

As in Section~\ref{sec:upper_bounds}, it is natural to consider the dual hypergraph $\cF$, whose vertex set is $E(\cH)$ and whose edges are given by the collections $\{e: e\ni v\}$ for each vertex $v\in V(\cH)$. Finding a spanning subhypergraph of $\cH$ whose degrees are all $1$ or $2$ corresponds to finding a subset of vertices $S\subseteq V(\cF)$ such that every hyperedge $e\in E(\cF)$ contains one or two of these vertices. In other words, it corresponds to finding a bicoloring of $V(\cF)$ where no edge of $\cF$ is monochromatic, which is saying precisely that $\cF$ has Property B.

To prove Proposition~\ref{prop:2-shallow}, we use a simple adaptation of Seymour's argument from \cite{Seymour74}, by exploiting the fact that $\cF$ has at least two (in fact, three) disjoint perfect matchings whenever $\cH$ is a tripartite hypergraph. The heart of the argument is the following claim.

\begin{claim}\label{claim:minimal}
Let $\cF$ be a hypergraph, and suppose that every proper subhypergraph of $\cF$ has property $B$. Let us assign a vector $x_v\in \bR^{|E(\cF)|}$ to each vertex of the hypergraph $\cF$, where $x_v(e)=1$ if $v\in e$ and $x_v(e)=0$ otherwise. If the vectors $x_v$ are linearly dependent, then $\cF$ has property $B$. In particular, if $|E(\cF)|<|V(\cF)|$, then $\cF$ has Property $B$. 
\end{claim}
\begin{proof}
If the vectors $x_v$ are linearly dependent, there exists a linear dependency $\sum_{v} \alpha_v x_v=0$ with not all $\alpha_v$ equal to zero. In other words, for every edge $e\in E(\cF)$ we have $\sum_{v\in e} \alpha_v=0$.

Let $R=\{v:\alpha_v<0\}$ and $B=\{v:\alpha_v>0\}$, and color the vertices of $R$ red and vertices of $B$ blue. Note that some vertices may have $\alpha_v=0$ and thus remain uncolored. Since $\sum_{v\in e} \alpha_v=0$, any edge containing a red vertex also must contain a blue vertex. Let $U$ be the set of uncolored vertices, and let $\cF[U]$ be the induced subhypergraph on this set. By assumption, $\cF[U]$ can be bicolored without monochromatic edges. Extending the coloring using the sets $R, B$ gives a bicoloring of $\cF$ without monochromatic edges, as we wanted.
\end{proof}

\begin{lemma}\label{lemma:property_B}
Let $\cF$ be a $3$-uniform hypergraph of maximum degree $3$ with two edge-disjoint perfect matchings. Then, $\cF$ has Property $B$. 
\end{lemma}
\begin{proof}
Suppose that $\cF$ does not have Property $B$, and let $\cC\subseteq \cF$ be the minimal subgraph of $\cF$ which also does not have Property $B$. 

To start, let us show that $|E(\cC)|\leq |V(\cC)|$. Indeed, note that the average degree of $\cC$ is at most $3$, since this is the maximum degree of $\cC$. On the other hand, the average degree of $\cC$ is at least ${3|E(\cC)|}/{|V(\cC)|}$, since $\cC$ is a $3$-uniform hypergraph. Combining these two bounds gives $|E(\cC)|\leq |V(\cC)|$.

On the other hand, since $\cC$ is a minimal subhypergraph not having Property $B$, by Claim~\ref{claim:minimal}, we must have $|E(\cC)|\geq |V(\cC)|$. Hence, all of the inequalities of the previous paragraph must be equalities, and in particular all vertices of $\cC$ must have degree $3$ in $\cC$. This means that no edge contains vertices of both $\cC$ and $\cF\backslash \cC$, and so every perfect matching of $\cF$ restricts to a perfect matching of $\cC$. Hence, $\cC$ has two distinct perfect matchings.

This implies that the vectors $\{x_v:v\in V(\cC)\}$ cannot be linearly independent, as follows. Consider the vertex-edge incidence matrix $A$ of the hypergraph $\cC$, which has the vectors $x_v$ as rows. If the rows of this square matrix are linearly independent, so are the columns. In other words, if $y_e\in \bR^{|V(\cC)|}$ is the indicator vector of the edge $e$, the vectors $\{y_e:e\in E(\cC)\}$ are linearly independent.

However, this is not the case: if $M, M'$ are the perfect matchings of $\cC$, then $\sum_{e\in M} y_e=\sum_{e\in M'} y_e$ is a nontrivial linear dependence between the columns of $A$ (both sums give an all-ones vector since $M, M'$ are perfect). Hence, the rows of $A$ must also be linearly dependent, which is what we aimed to show. By Claim~\ref{claim:minimal} this means that $\cC$ has property B, contradiction.
\end{proof}

\begin{proof}[Proof of Proposition~\ref{prop:2-shallow}.]
To prove Proposition~\ref{prop:2-shallow}, it suffices to observe that the dual $\cF$ of a tripartite $3$-uniform hypergraph $\cH$ has three disjoint perfect matchings. Indeed, the vertices inside each of the three partite sets of $\cH$ are incident to distinct edges, and thus they correspond to a perfect matching in the dual. Hence, Proposition~\ref{prop:2-shallow} follows directly by dualizing Lemma~\ref{lemma:property_B}.
\end{proof}

\noindent
{\bf Acknowledgment.}
Part of the research on this project was done during a visit of the first and last authors at the Institute of Mathematical Sciences, National University of Singapore in 2026. We thank our hosts at the Institute for their hospitality.  

\medskip
\noindent
{\bf AI disclosure.} 
We have tried to use AI tools to resolve Conjecture 1.5. Although this attempt was ultimately unsuccessful, it pointed us to Seymour's argument from \cite{Seymour74}, which resolves the 3-regular case and which we decided to include in the paper.

\appendix

\section{An alternative proof of \texorpdfstring{$f(3)\leq 4$}{f(3) ≤ 4}}

Here we provide an alternative proof that regular hypergraphs contain shallow edge hitting sets. We first present an argument which gives a subhypergraph with maximum degree $5$, and then modify it to show that $5$ can be reduced to $4$.

This proof is reminiscent of the proof of Petersen's theorem on 2-factors of even-regular graphs. However, since we are working with hypergraphs instead of graphs, we substitute Hall's theorem with a consequence of the Aharoni--Haxell hypergraph matching theorem \cite{AharoniHaxell00}. Before we state it, let us recall that a $3$-uniform hypergraph $\cH$ is called \textit{bipartite} if there is a partition $V(\cH)=A\cup B$ such that each edge $e\in E(\cH)$ has $|e\cap A|=1$ and $|e\cap B|=2$. Also, let us recall that the link graph of a set $S$ is the graph $L_S$ on the vertex set $V(\cH)$ in which $xy$ is an edge if and only if there exists $s\in S$ such that $\{x, y, s\}\in E(\cH)$.
We denote by $\nu(G)$ the maximum size of a matching in a graph $G$, and by $\tau(G)$ the minimum size of a vertex set that intersects every edge of $G$.

\begin{theorem}[Corollary 1.2 in \cite{AharoniHaxell00}]\label{thm:hypergraph_matching}
Suppose that $\cH$ is a bipartite $3$-uniform hypergraph. If every set $S\subseteq A$ has $\nu(L_S)\geq 2|S|$, then there exists a matching $M$ in $\cH$ saturating $A$.
\end{theorem}

\begin{proposition}\label{prop:5-shallow}
If $\cH$ is a nonempty regular $3$-uniform hypergraph, then $\cH$ has a spanning subhypergraph with degrees in the set $\{1, 2, 3, 4, 5\}$.
\end{proposition}
\begin{proof}
Consider the following auxiliary hypergraph $\cA$. The vertex set of $\cA$ consists of five copies of $V(\cH)$, i.e. for each vertex $v\in V(\cH)$ we add the vertices $v^{(1)}, v^{(2)}, v^{(3)}, v^{(4)}, v^{(5)}$ to $V(\cA)$. The edge set of $\cA$ is defined as follows: for each edge $e=\{x, y, z\}\in E(\cH)$, we add to $E(\cA)$ the following six edges
\begin{center}
\begin{tabular}{ccc}
$\{x^{(1)}, y^{(2)}, z^{(3)}\}$, & $\{y^{(1)}, z^{(2)}, x^{(3)}\}$, & $\{z^{(1)}, x^{(2)}, y^{(3)}\}$,\\
$\{x^{(1)}, y^{(4)}, z^{(5)}\}$, & $\{y^{(1)}, z^{(4)}, x^{(5)}\}$, & $\{z^{(1)}, x^{(4)}, y^{(5)}\}$.
\end{tabular}
\end{center}

Note that the resulting hypergraph $\cA$ is tripartite, since every edge contains one vertex from each of the three sets $A=\{v^{(1)}:v\in V(\cH)\}$, $B=\{v^{(2)}, v^{(4)}:v\in V(\cH)\}$ and $C=\{v^{(3)}, v^{(5)}:v\in V(\cH)\}$. See the illustration in Figure~\ref{fig:auxiliary}. Further, since $\cH$ is $d$-regular, every vertex of $A$ has degree $2d$, while every vertex of $B$ and $C$ has degree $d$. 

\begin{figure}[hbpt!]
  \begin{center}
\begin{tikzpicture}[scale=1,
  smallbox/.style={draw, rounded corners=2pt, minimum width=1.55cm, minimum height=2cm, inner sep=0pt},
  bigbox/.style={draw, rounded corners=2pt, minimum width=1.55cm, minimum height=5.2cm, inner sep=0pt},
  vertex/.style={circle, fill=black, inner sep=0pt, minimum size=5pt},
  vertexlabel/.style={ inner sep=1pt},
  auxedge/.style={line width=0.85pt}
]
  \node[smallbox,label=above:{$A$}] at (0,0) {};
  \node[bigbox,label=above:{$B$}] at (3,0) {};
  \node[bigbox,label=above:{$C$}] at (6,0) {};

  \coordinate (x1) at (0,0.62);
  \coordinate (y1) at (0,0);
  \coordinate (z1) at (0,-0.62);

  \coordinate (x2) at (3,2.0);
  \coordinate (y2) at (3,1.2);
  \coordinate (z2) at (3,0.4);
  \coordinate (x3) at (6,1.75);
  \coordinate (y3) at (6,1.15);
  \coordinate (z3) at (6,0.55);

  \coordinate (x4) at (3,-0.4);
  \coordinate (y4) at (3,-1.2);
  \coordinate (z4) at (3,-2.0);
  \coordinate (x5) at (6,-0.55);
  \coordinate (y5) at (6,-1.15);
  \coordinate (z5) at (6,-1.75);

  \draw[auxedge] (x1) -- (y2) -- (z3);
  \draw[auxedge] (y1) -- (z2) -- (x3);
  \draw[auxedge] (z1) -- (x2) -- (y3);

 \draw[auxedge] (x1) -- (y4) -- (z5);
  \draw[auxedge] (y1) -- (z4) -- (x5);
  \draw[auxedge] (z1) -- (x4) -- (y5);

  \foreach \name in {x1,y1,z1,x2,y2,z2,x3,y3,z3,x4,y4,z4,x5,y5,z5}
    \node[vertex] at (\name) {};

  \node[vertexlabel,left=2pt] at (x1) {$x^{(1)}$};
  \node[vertexlabel,left=2pt] at (y1) {$y^{(1)}$};
  \node[vertexlabel,left=2pt] at (z1) {$z^{(1)}$};
  \node[vertexlabel,above=2pt] at (x2) {$x^{(2)}$};
  \node[vertexlabel,above=2pt] at (y2) {$y^{(2)}$};
  \node[vertexlabel,above=2pt] at (z2) {$z^{(2)}$};
  \node[vertexlabel,right=2pt] at (x3) {$x^{(3)}$};
  \node[vertexlabel,right=2pt] at (y3) {$y^{(3)}$};
  \node[vertexlabel,right=2pt] at (z3) {$z^{(3)}$};
  \node[vertexlabel,below=2pt] at (x4) {$x^{(4)}$};
  \node[vertexlabel,below=2pt] at (y4) {$y^{(4)}$};
  \node[vertexlabel,below=2pt] at (z4) {$z^{(4)}$};
  \node[vertexlabel,right=2pt] at (x5) {$x^{(5)}$};
  \node[vertexlabel,right=2pt] at (y5) {$y^{(5)}$};
  \node[vertexlabel,right=2pt] at (z5) {$z^{(5)}$};
\end{tikzpicture}
\end{center}

\caption{The auxiliary hypergraph $\cA$.}
\label{fig:auxiliary}
\end{figure}

The goal is to show that there exists a matching covering $A$ in this auxiliary hypergraph. To show this, note that for any $S\subseteq A$, the link graph $L_S$ is a bipartite multigraph between $B$ and $C$ with maximum degree $d$ and $2|S|d$ edges, since every vertex of $S$ contributes at least $2d$ edges to the link graph. Hence, the cover number of this bipartite graph must be at least $2|S|$, and so by K\H{o}nig's theorem we have $\nu(L_S)=\tau(L_S)\geq 2|S|$. By Theorem~\ref{thm:hypergraph_matching}, this shows that $\cA$ contains a matching $M$ saturating $A$.

In other words, for every $v\in V(\cH)$, there exists a unique edge of $M$ containing $v^{(1)}$, say $\{v^{(1)}, u^{(2)}, w^{(3)}\}$ (or $\{v^{(1)}, u^{(4)}, w^{(5)}\}$). We construct the subhypergraph $\cC$ by adding the edge $\{v, u, w\}$ for each $v$. We claim that every vertex appears in at most $5$ edges of $\cC$. Indeed, if a vertex $v$ appears in an edge of $\cC$, it must be that one of its copies in $\cA$ was present in one of the edges of the matching $M$. But there are $5$ copies of $v$ in $\cA$, and so $v$ can appear in at most $5$ edges of $\cC$. This shows that $\cC$ has maximum degree at most $5$, as we wanted. 
\end{proof}

\begin{proposition}\label{prop:4-shallow}
If $\cH$ is a nonempty regular $3$-uniform hypergraph, then $\cH$ has a spanning subhypergraph with degrees in the set $\{1, 2, 3, 4\}$.
\end{proposition}
\begin{proof}
We start from the subhypergraph $\cC$ obtained in Proposition~\ref{prop:5-shallow} and remove some edges to reduce its maximum degree by one. Note that we treat $\cC$ as a multihypergraph.

The main observation is that if a vertex $v$ has $\deg_{\cC}v=5$, then we may delete the edge containing $v^{(1)}$ in the matching $M$ (which we included in order to cover $v$) and reduce the degree of $v$ in this way. However, one must be careful in which order the vertices $v$ are chosen, in order to ensure that every vertex is still covered after several edges get deleted in this process.

We call the edge containing $v^{(1)}$ in the matching $M$ the \textit{special edge} of $v$, and we use the same term to refer to the corresponding edge in $\cC$.

Let $D$ be an auxiliary directed graph on the set of vertices of degree $5$. We put a directed edge in $D$ from $u$ to $v$ if $v$ is contained in the special edge of $u$ and $v\neq u$. The goal is to delete some special edges associated to the vertices of $D$, and we will use $D$ to guide the order in which we do this.

The maximum outdegree in $D$ is $2$, and hence there exists a vertex $v$ of indegree at most $2$. We now delete the special edge $e_v$ of the vertex $v$ from $\cC$. 

We claim that even though $e_v$ was deleted, $v$ will still be covered at the end of this procedure. This is simple to see: since $v$ had at most $2$ inedges, at most two additional edges covering $v$ will be deleted. Since $v$ had appeared in $5$ edges of $\cC$, even after all the deletions, $v$ will still be covered. 

We further claim that no other vertex $w$ of $e_v$ is left uncovered. If $w$ never had degree $5$, then the special edge $e_w$ will not be deleted from $\cC$ and $w$ will remain covered. If $w$ had degree $5$ and the special edge $e_w$ was deleted in the past, we have just argued in the previous paragraph that $w$ will remain covered.

We iterate this process on the remaining vertices that are covered at least $5$ times, until all vertices are covered at most $4$ times, which completes the proof.
\end{proof}

\end{document}